\documentclass[10pt]{amsart}
\usepackage{amsmath}
\usepackage{stmaryrd}
\usepackage{amscd,amssymb,verbatim}
\usepackage{amssymb,amsmath,amsthm,epsfig}

\newcommand{\N}{\mathbb{N}}
\newcommand{\I}{\mathbb{I}}

\newcommand{\R}{\mathbb{R}}

\newtheorem*{thmA}{Theorem A}
\newtheorem*{thmB}{Theorem B}
\newtheorem*{thmC}{Theorem C}
\newtheorem{thm}{Theorem}[section]
\newtheorem{lem}[thm]{Lemma}
\newtheorem{cor}[thm]{Corollary}
\newtheorem{prop}[thm]{Proposition}
\theoremstyle{definition}
\newtheorem{defin}[thm]{Definition}
\newtheorem{ex}[thm]{Example}
\newtheorem{remark}[thm]{Remark}

\newcommand{\vp}{\varepsilon}

\newcommand{\ub}{\overline{u}}

\newcommand{\kin}{\!\in\!}

\newcommand{\kminus}{\!-\!}

\newcommand{\coo}{c_{00}}

\def\hangbox to #1 #2{\vskip1pt\hangindent #1\noindent \hbox to
#1{#2}$\!\!$}

\begin{document}

\title{On Shrinking and Boundedly Complete Schauder Frames of Banach spaces }
\author{Rui Liu}
\address{Department of Mathematics and LPMC, Nankai
University, Tianjin 300071, P.R. China}
\address{Department of
Mathematics, Texas A$\&$M University, College Station, Texas
77843-3368}

\email{rliu@math.tamu.edu; leorui@mail.nankai.edu.cn}
\thanks{This work is supported by funds from the A. G. $\&$ M. E. Owen
Chair of Mathematics at Texas A$\&$M University, the China
Scholarship Council (CSC), the National Natural Science Foundation
of China (No. 10571090), and the Doctoral Programme Foundation of
Institution of Higher Education (No. 20060055010)}

\begin{abstract}
This paper  studies Schauder frames in Banach
spaces, a concept which is a natural generalization of frames in Hilbert
spaces and Schauder bases in Banach spaces. The associated minimal and maximal spaces
 are introduced, as are shrinking and boundedly complete Schauder frames. Our main results
extend the classical duality theorems on bases to the situation of Schauder frames.
 In particular, we will generalize James' results on shrinking and boundedly complete bases to
  frames. Secondly we will extend his characterization of the reflexivity  of spaces with unconditional
   bases to  spaces with unconditional frames.
\end{abstract}

\maketitle

\section{Introduction}

The theory of frames in Hilbert spaces presents  a central tool in
many areas and  has developed rather rapidly in the past decade. The
motivation has come from applications to engineering, i.e. signal
analysis, as well as from applications to different areas of
Mathematics, such as, sampling theory \cite{AG}, operator theory
\cite{HL}, harmonic analysis \cite{Gr}, nonlinear sparse
approximation \cite{DE}, pseudo-differential operators \cite{GH},
and quantum computing \cite{EF}. Recently, the theory of frames also
showed connections to theoretical problems such as the
Kadison-Singer Problem \cite{CFTW}.

A standard frame for a Hilbert space $H$ is a family of vectors
$x_i\in H$,  $i\in\N$, such that there are constants $A, B>0$ for
which
$$A\,\|x\|^2\le \sum|\langle x, x_i\rangle|^2\le B\,\|x\|^2, \mbox{ whenever } x\in H.$$
In this paper we consider Schauder frames in Banach spaces,
 which, on the one hand, generalize Hilbert frames, and extend the notion of
 Schauder basis, on the other.

In \cite{CL}, D. Carando and S. Lassalle consider the duality theory
for atomic decompositions. In our independent work, we will mostly
concentrate on properties of Schauder frames, which do not depend on
the choice of associated spaces, give out the concepts of minimal
and maximal (associated) spaces and the corresponding minimal and
maximal (associated) bases with respect to Schauder frames, and
closely connect them to the duality theory. Moreover, we extend
James' well known results on characterizing the reflexivity of
spaces with an unconditional bases, to spaces with unconditional
frames.

In Section 2 we recall the basic definitions and properties of
Schauder frames. Then we introduce the concept of  shrinking
and boundedly complete frames and prove some elementary facts.

Section 3 deals with the concept of associated spaces, and
introduces the definitions of minimal and maximal (associated) spaces
and the corresponding minimal and maximal (associated) bases with
respect to Schauder frames.

In Section 4 we extend James' results on shrinking and boundedly
bases to frames \cite{Ja}   and prove the following  theorems. All
necessary definitions can be found in the following sections   2 and
3.

\begin{thmA} Let $(x_i,f_i)\subset X\times X^*$ be a Schauder frame of a Banach space $X$  and assume that
for all $m\in\N$
$$\lim_{n\to\infty} \|f_m|_{\text{span}(x_i:i\ge n)}\|=0.$$
Then the following are equivalent.
\begin{enumerate}
\item $(x_i,f_i)$ is shrinking.
\item Every normalized block of $(x_i)$ is weakly null.
\item $X^*=\overline{\text{span}(f_i:i \in\N)}$.
\item The minimal associated  basis is shrinking.
\end{enumerate}
\end{thmA}

\begin{thmB} Let $(x_i,f_i)\subset X\times X^*$ be a Schauder frame of a Banach space $X$  and assume that
for all $m\in\N$
$$\lim_{n\to\infty} \|f_m|_{\text{span}(x_i:i\ge n)}\|=0 \text{ and } \lim_{n\to\infty} \|x_m|_{\text{span}(f_i:i\ge n)}\|=0.$$
Then the following are equivalent.
\begin{enumerate}
\item $(x_i,f_i)$ is boundedly complete.
\item $X$ is isomorphic to $ \overline{\text{span}(f_i:i \in\N)}^*$ under the natural canonical map.
\item The maximal associated basis is boundedly complete.
\end{enumerate}
\end{thmB}

In Section 5, we discuss
unconditional Schauder frames. We obtain a generalization of James's
theorem and prove that a Banach space with a locally shrinking and
unconditional Schauder frame is either reflexive or contains
isomorphic copies of $\ell_1$ or $c_0$.

\begin{thmC}  Let $(x_i,f_i)\subset X\times X^*$ be a unconditional Schauder frame of a Banach space $X$  and assume that
for all $m\in\N$
$$\lim_{n\to\infty} \|f_m|_{\text{span}(x_i:i\ge n)}\|=0.$$
Then $X$ is reflexive  if an only if $X$ does not contain isomorphic copies of $c_0$ and $\ell_1$
\end{thmC}

All Banach spaces in this paper are considered to be spaces over the real number field
$\R$. The unit sphere and the unit ball of a
Banach space $X$  are denoted by $S_X$ and $B_X$, respectively. The vector space
of scalar sequences $(a_i)$, which vanish eventually, is  denoted by $c_{00}$. The usual unit
vector basis of $c_{00}$, as well as the unit vector basis of $c_0$ and
$\ell_p$ ($1\le p<\infty$) and the corresponding coordinate
functionals will be denoted by $(e_i)$ and $(e^*_i)$, respectively.

Given two sequences $(x_i)$ and $(y_i)$ in some Banach space, and
given a constant $C>0$, we say that \emph{$(y_i)$ $C$-dominates
$(x_i)$}, or that \emph{$(x_i)$ is $C$-dominated by $(y_i)$}, if
\begin{eqnarray*}
\Big\|\sum a_i x_i\Big\|\le C\Big\|\sum a_i y_i\Big\| \quad \mbox{
for all } (a_i)\in\coo.
\end{eqnarray*}
We say that \emph{$(y_i)$ dominates $(x_i)$}, or that \emph{$(x_i)$
is dominated by $(y_i)$},
$(y_i)$ $C$-dominates $(x_i)$ for some  constant $C>0$.

\section{Frames in Banach Spaces}

In this section, we give a short review of the concept of frames in
Banach spaces, and make some preparatory observations.

\begin{defin}
Let $X$ be a (finite or infinite dimensional) separable Banach
space. A sequence $(x_i,f_i)_{i\in \I}$, with $(x_i)_{i\in
\I}\subset X$ and $(x_i)_{i\in \I}\subset X^*$ with $\I=\N$ or
$\I=\{1,2, .\, .\, .\, , N\}$ for some $N\in\N$, is called a {\em
(Schauder) frame of $X$} if for every $x\in X$,
\begin{equation}\label{def.frame.eq1}
x=\displaystyle\sum_{i\in\I} f_i(x) x_i.
\end{equation}
In case that $\I=\N$, we mean that the series in
(\ref{def.frame.eq1}) converges in norm, that is,
\begin{equation}
x=\|\cdot\|-\displaystyle\lim_{n\rightarrow\infty}\sum_{i=1}^n
f_i(x)x_i.
\end{equation}

An \emph{unconditional frame of $X$} is a frame $(x_i,f_i)_{i\in\N}$
for $X$ for which the convergence in (\ref{def.frame.eq1}) is
unconditional.

We call a frame $(x_i,f_i)$ {\em bounded } if
$$\sup_i\|x_i\|<\infty \, \mbox{ and } \, \sup_i\|f_i\|<\infty,$$
and \emph{semi-normalized} if $(x_i)$ and $(f_i)$ both are
semi-normalized, that is, if
$$0<\inf_i\|x_i\|\le \sup_i\|x_i\|<\infty \, \text{ and } \,
0<\inf_i\|f_i\|\le \sup_i\|f_i\|<\infty.$$
\end{defin}

\begin{remark} Throughout this paper, it will be our convention that we only
consider non-zero frames $(x_i,f_i)$ indexed by $\N$, that is, the index
set $\I$ will always be $\N$ and we assume that $x_i\neq 0$ and
$f_i\neq 0$ for all  $i\in\N$.
\end{remark}

In the following proposition we recall some easy observations from
\cite{CHL} and \cite{CDOSZ}.
\begin{prop}\cite{CHL,CDOSZ}\label{basic.pp1}
Let $(x_i,f_i)$  be a frame of $X$.
\begin{enumerate}
\item[(a)]
\begin{enumerate}
\item[(i)] Using the Uniform Boundedness Principle we deduce that
$$K=\sup_{x\in B_X}\sup_{m\le n} \Big\|\sum_{i=m}^n f_i(x) x_i\Big\|<\infty.$$
We call $K$ the \textrm{projection constant of $(x_i,f_i)$}.
\item[(ii)]
If $(x_i,f_i)$ is an unconditional frame, then it also follows from
the Uniform Boundedness Principle that
$$K_u=\sup_{x\in B_X}\sup_{(\sigma_i)\subset\{\pm1\}} \Big\|\sum \sigma_i f_i(x) x_i\Big\|<\infty.$$
We call $K_u$ the \textrm{unconditional constant} of $(x_i,f_i)$.
\end{enumerate}
\item[(b)] The sequence $(f_i,x_i)$ is a $w^*$-Schauder frame of $X^*$, that is to say, for every $f\in
X^*$,
$$f=w^*-\lim\limits_{n\to\infty}\sum_{i=1}^n f(x_i) f_i.$$
\item[(c)] For any  $f\in X^*$ and $m\le n$ in $\N$, we have
\begin{equation}\label{E:2.5a.1}
\Big\| \sum_{i=m}^n f(x_i) f_i\Big\|= \sup_{x\in
B_X}\Big|\sum_{i=m}^n f(x_i) f_i(x)\Big|\le
 \|f\| \sup_{x\in B_X}\Big\|\sum_{i=m}^n f_i(x) x_i\Big\|\le K \|f\|,
\end{equation}
and
\begin{align}\label{E:2.5a.2}
\Big\| \sum_{i=m}^n f(x_i) f_i\Big\|&=\sup_{x\in
B_X}\Big|\sum_{i=m}^n f(x_i) f_i(x)\Big|
=\sup_{x\in B_X}\Big|f\Big(\sum_{i=m}^n  f_i(x) x_i\Big)\Big|\\
& \le \sup\limits_{ z \in \, \text{\rm span}(x_i:i\ge m), \|z\|\le K
}|f(z)|=
 K\|f|_{\text{\rm span}(x_i:i\ge m)}\|,\notag
\end{align}
where $K$ is the projection constant of $(x_i,f_i)$.
\end{enumerate}
\end{prop}
Next, we present some  basic properties of frames in Banach spaces.
\begin{prop}\label{norming}
Let $(x_i, f_i)$ be a frame of a Banach space $X$ and denote by $K$
the projection constant of $(x_i, f_i)$. Then
${\overline{\textrm{span}}(f_i:i\in\N)}$ is a norming subspace of
$X^*$.
\end{prop}
\begin{proof}
By Proposition \ref{basic.pp1} (b) and (c) (\ref{E:2.5a.1}), for all
$f\in B_{X^*}$ and $n\in \N$ we have
$$f=w^*\mbox{-}\lim\limits_{n\to\infty}\sum\limits_{i=1}^n f(x_i)
f_i, \ \ \left\|\sum_{i=1}^n f(x_i) f_i\right\|\le K,$$ where $K$ is
the projection constant of $(x_i,f_i)$.
Thus, we obtain that $$B_{X^*}\subset \overline{K\cdot
B_{X^*}\bigcap\mathrm{span}(f_i:i\in\N)}^{w^*}\subset K\cdot
B_{X^*}.$$ Then it is easy to deduce that
$\overline{\textrm{span}}(f_i:i\in\N)$ is norming for $X$.
\end{proof}
%%%%%%%%%%%%%%%%%%%%%%%%%%%%%%%%%%%%%%%%%%%%%%%%%%%%%%%%%%%%%%%%%%%%%%%
\begin{defin}\label{def.local}
Let $(x_i,f_i)$ be a frame of a Banach space $X$.

 $(x_i,f_i)$ is called \emph{locally shrinking} if for all $m\in\N$
$\|f_m|_{\textrm{span}(x_i:i\geq n)}\|\to0$ as $n\to \infty$.
 $(x_i,f_i)$ is called \emph{locally boundedly complete} if for all
$m\in\N$ $\|x_m|_{\textrm{span}(f_i:i\geq n)}\|\to0$ as $n\to
\infty$.  $(x_i,f_i)$ is called {\em weakly localized } if it
is locally shrinking and locally boundedly complete.

The frame $(x_i,f_i)$ is called {\em  pre-shrinking } if
$(f_i,x_i)$ is a
frame of $X^*$.It is called \emph{pre-boundedly complete} if for all $x^{**}\in X^{**}$, $\sum\limits_{i=1}^\infty
x^{**}(f_i) x_i$ converges.

We call $(x_i,f_i)$ {\em shrinking } if it is  locally shrinking and  pre-shrinking, and we call
$(x_i,f_i)$ {\em boundedly complete } if it  weakly localized and {\em pre boundedly complete}.
 \end{defin}

It is clear that every basis for a Banach space is
 weakly localized. However, it is
false for frames. The following example is an unconditional and
semi-normalized frame  for $\ell_1$ which is not locally shrinking
or locally boundedly complete. We leave the proof to the reader.
\begin{ex}
Let $(e_i)$ denote the usual unit vector basis of $\ell_1$ and let
$(e_i^*)$ be the corresponding coordinate functionals, and set
${\bf 1}=(1,1,1, .\, .\, .\, )\in \ell_\infty$. Then define a
sequence $(x_i,f_i)\subset \ell_1 \times \ell_\infty$ by putting
$x_{2i-1}=x_{2i}=e_i$ for all $i\in\N$ and
$$
f_i=\left\{%
\begin{array}{ll}
    {\bf 1}, & \hbox{if $\!i=\!1$;} \\
    e_1^*-{\bf 1}, & \hbox{if $i\!=\!2$;} \\
    e_k^*-e_1^*/2^k, & \hbox{if $i\!=\!2k\kminus1$ for $k\kin\N\setminus\{1\}$;} \\
    e_1^*/2^k, & \hbox{if $i\!=\!2k$ for $k\kin\N\setminus\{1\}$.} \\
\end{array}%
\right.$$
\end{ex}

\begin{prop}\label{self}
Let $(x_i,f_i)$ be a frame of a Banach space $X$. Then the space
$$X_0=\Big\{x\in X: \|x|_{\textrm{span}(f_i:i\geq n)}\|\to 0 \mbox{ as } n\to\infty \Big\}$$
is a norm closed subspace of $X$. Moreover, if  $(x_i,f_i)$ is locally
boundedly complete, then $X_0=X.$
\end{prop}
\begin{proof}
If $(x_k)\subset X_0$ with $x_k\to x$ in $X$, then given any
$\vp>0$, there are $k_0$ with $\|x-x_{k_0} \|\le \vp$, and
$n_0\in\N$ such that for all $n\ge n_0$,
$$\|x|_{\textrm{span}(f_i:i\geq n)}\| \le \|x-x_{k_0}
\|+\|x_{k_0}|_{\textrm{span}(f_i:i\geq n)}\|\le 2\vp,$$ which
 implies that $x\in X_0.$

If $(x_i,f_i)$ is locally boundedly complete, then $x_i\in X_0$ for
all $i\in\N$. It follows that
$X=\overline{\textrm{span}}(x_i:i\in\N)\subset X_0$. Thus, we
complete the proof.
\end{proof}

\begin{prop}\label{P:2.5b} Let $(x_i,f_i)$  be a frame of a Banach space
$X$. Then the space
$$Y=\Big\{f\in X^*: f=\|\cdot\|-\lim_{n\to\infty}\sum_{i=1}^n f(x_i) f_i\Big\},$$
is a norm closed subspace of $X^*$. Moreover, if $(x_i,f_i)$ is
locally shrinking, then $$Y=\overline{\mathrm{span}}(f_i:i\in\N),$$
and, thus, $(f_i,x_i)$ is a frame for $Y$.
\end{prop}
\begin{proof} First, define a new norm
$|\!|\!|\cdot|\!|\!|$ on $X^*$ as follows
$$|\!|\!|f|\!|\!|=\sup\limits_{m\le n}\Big\| \sum\limits_{i=m}^n
f(x_i) f_i\Big\| \quad \mbox{ for all } f\in X^*.$$  By Proposition
\ref{basic.pp1} (c) this is an equivalent norm of $(X^*,
\|\cdot\|)$. Thus, if $(g_k)\subset Y$ with $g_k\to g$ in $X^*$, it
follows that,
$$\lim_{k\to\infty}|\!|\!|g-g_k |\!|\!| =\lim_{k\to\infty}\sup_{m\le n}\Big\| \sum_{i=m}^n g(x_i) f_i -    \sum_{i=m}^n g_k(x_i)
  f_i\Big\|=0.$$
Thus, given any $\vp>0$, there are $k_0$ with $|\!|\!|g-g_{k_0}
|\!|\!|\le \vp$, and $m_0\in\N$ such that for all $n\ge m\ge m_0$,
$\Big\| \sum\limits_{i=m}^n g_{k_0}(x_i) f_i\Big\|\le\vp$, and thus,
  $$\Big\| \sum_{i=m}^n g(x_i) f_i\Big\|\le |\!|\!|g-g_{k_0} |\!|\!|+\Big\| \sum_{i=m}^n g_{k_0}(x_i) f_i\Big\|
   \le 2\vp,$$
which implies that $\sum\limits_{i=1}^\infty g(x_i) f_i$ converges.
By Proposition \ref{basic.pp1} (b), we get
$g=\sum\limits_{i=1}^\infty g(x_i) f_i\in Y$.

If $(x_i,f_i)$ is locally shrinking, it follows from Proposition
\ref{basic.pp1} (c) that  for all $i\in\N$, $f_i\in Y$. Hence
$\overline{\text{\rm span}}(f_i:i\in\N)\subset Y$. On the other hand, it
is clear from the definition of $Y$ that
$Y\subset\overline{\text{\rm span}}(f_i:i\in\N)$. Therefore,
$Y=\overline{\textrm{span}}(f_i:i\in\N).$
\end{proof}

\section{Associated Spaces}

\begin{defin}\label{ass.def} Let $(x_i,f_i)$ be
a frame of a Banach space $X$ and let $Z$ be a Banach space with a
basis $(z_i)$. We call $Z$ an {\em associated space to $(x_i,f_i)$}
and $(z_i)$ an \emph{associated basis}, if
\begin{align*}
S:Z\to X,\quad \sum a_i z_i\mapsto \sum  a_i x_i \ \ \text{ and } \
\ T:X\to Z,\quad x=\sum f_i(x) x_i\mapsto \sum  f_i(x) z_i,
\end{align*}
are bounded operators. We call $S$ the {\em associated
reconstruction operator} and $T$ the {\em associated decomposition
operator} or \emph{analysis operator}.
\end{defin}

\begin{remark}
If $(x_i,f_i)$ is a frame of a Banach space $X$ and $Z$ a
corresponding associated space with an associated basis $(z_i)$,
then (see \cite[Definition 2.1]{CHL} or \cite{Ch}) $(x_i,f_i)$ is an \emph{atomic
decomposition} of $X$ with respect to $Z$. In our paper, we will
mostly concentrate on frames and properties which are independent of
the associated spaces.
\end{remark}

\begin{prop}\label{P:2.7a}
Let $(x_i,f_i)$ be a frame of a Banach space $X$ and let $Z$ be an
associated space with an associated basis $(z_i)$. Let $S$ and $T$
be the associated reconstruction operator and the associated
decomposition operator, respectively.

Then $S$ is a surjection onto $T(X)$, and $T$ is an isomorphic
embedding from $X$ into $Z$. Moreover, for all $i\in\N$,
$S(z_i)=x_i$ and\, $T^*(z^*_i)=f_i$.
\end{prop}
\begin{proof}
Note that for any $x\in X$, it follows that
$$S\circ T(x)=S\circ T\Big(\sum f_i(x) x_i\Big)=S\Big(\sum f_i(x) z_i\Big)=
\sum f_i(x) x_i=x.$$ Therefore, $T$ must be an isomorphic embedding
and $S$ a surjection onto the space $T(X)=\Big\{ \sum f_i(x) z_i :
x\in X\Big\}$. And the map $P: Z\to Z$, $z\mapsto T\circ S(z)$ is a
projection onto $T(X)$. By Definition \ref{ass.def}, it is clear
that $S(z_i)=x_i$ for all $i\in\N$. Secondly, it follows that for
any $x\in X$ and $i\in\N$,
$$T^*(z^*_i)(x)=z^*_i \circ T\Big(\sum f_j(x) x_j\Big)= z^*_i\Big(\sum f_j(x) z_j\Big)= f_i(x),$$
and thus, $T^*(z^*_i)=f_i$, which completes our claim.
\end{proof}

We now introduce the notion of minimal bases.

\begin{defin}\label{minimal.def}
Let $(x_i)$  be a  non zero sequence in a Banach space $X$.

Define a norm on $\coo$ as follows
\begin{equation}\label{minimal.def.eq1}
\Big\|\sum a_i e_i\Big\|_{Min}=\max_{m\le n} \Big\|\sum_{i=m}^n a_i
x_i\Bigr\|_X \quad \mbox{ for all } \sum a_i e_i\in\coo,
\end{equation}
Denote by $Z_{Min}$ the completion of $c_{00}$ endowed with the norm
$\|\cdot\|_{Min}$. It is easy to prove that $(e_i)$, denoted by
$(e_i^{Min})$, is a bi-monotone basis of $Z_{Min}$. By the following
Theorem \ref{minimal.1} (b), we call $Z_{Min}$ and $(e_i^{Min})$ the
\emph{minimal space} and the \emph{minimal basis with respect to
$(x_i)$}, respectively.

Note that the operator:
$$S_{Min}:Z_{Min}\to X, \quad \sum a_i z_i\mapsto \sum a_i z_i,$$
is linear and bounded with $\|S_{Min}\|=1$

If $(x_i,f_i)$ is a frame the {\em minimal space } (or the {\em
minimal basis) with respect to $(x_i,f_i)$} is the  minimal space
(or the minimal basis) with respect to $(x_i)$.
\end{defin}

As the following result from \cite[Theorem 2.6]{CHL} shows,
associated spaces
       always exist.
\begin{thm}\cite[Theorem 2.6]{CHL}\label{minimal.1}
Let $(x_i,f_i)$ be a frame of a Banach space $X$ and let $Z_{Min}$
be the minimal space with the minimal basis $(e_i^{Min})$.
\begin{enumerate}
\item[(a)]
$Z_{\textrm{Min}}$ is an associated space to $(x_i,f_i)$ with the
associated basis $(e_i^{Min})$.
\item[(b)]
For any associated space $Z$ with an associated basis $(z_i)$,
$(e_i^{Min})$ is dominated by $(z_i)$.
\end{enumerate}

Thus, we will call $Z_{Min}$ and $(e_i^{Min})$ the \emph{minimal
associated space} and the \emph{minimal associated basis} to
$(x_i,f_i)$, respectively.

\end{thm}

We give a sketch of the proof.
\begin{proof} (a) Let $K$ be the projection constant of $(x_i,f_i)$. It follows that
  the map $T_{Min}: X\to Z_{Min}$
defined by
$$T_{Min}: X\to Z_{Min}, \quad x=\sum f_i(x) x_i\mapsto \sum f_i(x)
e_i^{Min},$$ is well-defined, linear and bounded and $\|T\|\le K$.
As already noted in Definition \ref{minimal.def} , the operator $S_{Min}: Z\to X$ is
linear and bounded.

(b) If $Z$ is an associated space with an associated basis $(z_i)$
and $S: Z\to X$ is the corresponding associated reconstruction
operator, then it follows that for any $(a_i)\in\coo$,
\begin{eqnarray}\label{minimal.remark1.eq1}
\Big\|\sum a_i e_i^{Min}\Big\|&=&\max_{m\le n}\Big\|\sum_{i=m}^n a_i
x_i\Big\|=\max_{m\le n}\Big\|\sum_{i=m}^n a_i S(z_i)\Big\|\\
&\le& \|S\| \max_{m\le n}\Big\|\sum_{i=m}^n a_i z_i\Big\|\le
K_Z\|S\| \Big\|\sum a_i z_i\Big\|,\nonumber
\end{eqnarray}
where $K_Z$ is the projection constant of $(z_i)$.
\end{proof}

Next we introduce the notion of the maximal space and the maximal
basis.
\begin{defin}\label{max.def}
Let $(x_i,f_i)$ be a frame of a Banach space $X$.

Define a norm on $\coo$ as follows
\begin{equation}\label{max.def.eq1}
\Big\|\sum a_i e_i\Big\|_{Max}=\sup_{\stackrel{(b_i)\in
c_{00}}{\max\limits_{m\leq n}\|\sum\limits_{i=m}^n b_i f_i\|\leq
1}}\Big|\sum a_i b_i\Big| \quad \mbox{ for all } \sum a_i
e_i\in\coo.
\end{equation}
Denote by $Z_{Max}$ the completion of $c_{00}$ under
$\|\cdot\|_{Max}$. Clearly, $(e_i)$ is a bi-monotone basis of
$Z_{Max}$, which will be denoted by $(e_i^{Max})$. We  call
$Z_{Max}$ and $(e_i^{Max})$ the \emph{maximal space} and the
\emph{maximal basis with respect to $(x_i,f_i)$}, respectively.
\end{defin}

\begin{thm}\label{associated.maximal.pp1}
Let $(x_i, f_i)$ be a frame of a Banach space $X$

and let $Z_{Max}$ be the maximal space with the maximal basis
$(e_i^{Max})$.
\begin{enumerate}
\item[(a)] If $Z$ is an associated space with an associated
basis $(z_i)$, then $(e_i^{Max})$ dominates $(z_i)$.
\item[(b)] The mapping
\begin{equation}\label{max.S.eq}
S_{Max}:Z_{Max}\to X,\quad z=\sum a_i e_i^{Max}\mapsto \sum  a_i
x_i,
\end{equation}
is well-defined, linear and bounded.
\item[(c)] If $(x_i,f_i)$ is
locally boundedly complete, then $Z_{\textrm{Max}}$ is an associated
space to $(x_i,f_i)$ with the associated basis $(e_i^{Max})$.

In this case, we call $Z_{Max}$ and $(e_i^{Max})$ the \emph{maximal
associated space} and the \emph{maximal associated basis to $(x_i,
f_i)$}.
\end{enumerate}

\end{thm}
\begin{proof} (a)  Let $Z$ be an associated space with an associated
basis $(z_i)$, $(z_i^*)$ is the corresponding coordinate
functionals, and  let $T:X\to Z$ be the associated decomposition operator.
By Proposition \ref{P:2.7a}
$T^*(z_i^*)=f_i$,  for all $i\in\N$. Thus, for any $(a_i)\in c_{00}$, we have
\begin{eqnarray}\label{max.eq2}
\Big\|\sum a_i z_i\Big\|&\leq& K_Z \sup_{\stackrel{(b_i)\in
c_{00}}{\|\sum b_i z_i^*\|\leq 1}}\Big|\Big\langle\sum a_i z_i, \sum
b_i z_i^* \Big\rangle\Big| \\
&\leq& K^2_Z \sup_{\stackrel{(b_i)\in c_{00}}{\max\limits_{m\leq
n}\|\sum\limits_{i=m}^n
b_i z_i^*\|\leq 1}}\Big|\sum a_i b_i\Big|\nonumber\\
&\leq& K^2_Z \sup_{\stackrel{(b_i)\in c_{00}}{\max\limits_{m\leq
n}\|T^*(\sum\limits_{i=m}^n b_i z_i^*)\|\leq \|T^*\|}}\Big|\sum a_i
b_i\Big| \nonumber\\
&\leq& K^2_Z \ \ \|T^*\| \sup_{\stackrel{(b_i)\in
c_{00}}{\max\limits_{m\leq n}\|\sum\limits_{i=m}^n b_i f_i\|\leq
1}}\Big|\sum a_i b_i\Big|
\leq K^2_Z \ \ \|T^*\| \ \ \Big\| \sum a_i e_i^{Max}
\Big\|,\nonumber
\end{eqnarray}
where $K_Z$ is the projection constant of $(z_i, z_i^*)$.

(b) Let $(Z_{Min}, (e_i^{Min}))$ be the minimal space to $(x_i,f_i)$
and by Theorem \ref{minimal.1} (a) let $T_{Min}: X\to Z_{Min}$ be
the corresponding associated decomposition operator. Then by
(\ref{max.eq2}), for any $(a_i)\in c_{00}$, we have
\begin{eqnarray}\label{max.eq4}
\max_{m\leq n}\Big\|\sum_{i=m}^na_i x_i\Big\|&=& \Big\|\sum a_i
e_i^{Min}\Big\|\leq C \Big\|\sum a_i e_i^{Max}\Big\|,
\end{eqnarray}
where $C=K^2_{Min} \, \|T^*_{Min}\|$ and $K_{Min}$ is the projection
constant of $(e_i^{Min})$. Thus, the map $S_{Max} :
Z_{Max}\rightarrow X$ with $S_{Max}(e_i^{Max})=x_i$, for $i\in\N$, is
well defined, linear and bounded with $\|S_{Max}\|\leq K^2_{Min}
\|T^*_{Min}\|$.

(c) If $(x_i,f_i)$ is locally boundedly complete, then
 for any $x\in X$ and $l\leq r$, we have
\begin{eqnarray*}
\Big\|\sum_{i=l}^r f_i(x) e_i^{Max}\Big\|=\sup_{\stackrel{(b_i)\in
c_{00}}{\max\limits_{m\leq n}\|\sum\limits_{i=m}^n b_i f_i\|\leq
1}}\Big|\sum_{i=l}^r b_i f_i(x)\Big|\leq \|x|_{\text{\rm span}(f_i : i\geq
l)}\|,
\end{eqnarray*}
which by Proposition \ref{self}, tends to zero as $l\to\infty$. Thus, the map
\begin{equation}
T_{Max} : X\rightarrow Z_{Max}, \ \ x=\sum f_i(x) x_i \mapsto \sum
f_i(x) e_i^{Max},
\end{equation}
is well-defined, linear and bounded with $\|T_{Max}\|\le1$, which
completes our proof.
\end{proof}

The following result emphases that for every frame, that
 associated bases dominate $(e^{Min}_i)$ and are
 dominated by $(e^{Max}_i)$.

\begin{cor}
Let $(x_i,f_i)$ be a frame of a Banach space $X$. Assume that
$(e_i^{Min})$ and $(e_i^{Max})$ is the minimal basis and the maximal
basis with respect to $(x_i,f_i)$, respectively. Then for any
associated space $Z$ with an associated basis $(z_i)$, there are
$C_1,C_2>0$ such that
\begin{equation}
C_1 \Big\|\sum a_i e_i^{Min}\Big\|\leq \Big\|\sum a_i z_i\Big\|\leq
C_2 \Big\|\sum a_i e_i^{Max}\Big\| \quad \mbox{for all \,$(a_i)\in
c_{00}$}.
\end{equation}
\end{cor}

\section{Applications of frames to duality theory}

The following  results extend James'  work on shrinking  and boundedly complete bases \cite{Ja} from
  to frames. Theorem \ref{shrinking.1} obviously yields  Theorem A and
  Theorem \ref{cor.bdd} implies Theorem B.

\begin{thm}\label{shrinking.1}
Let $(x_i,f_i)$ be a Schauder frame of a Banach space $X$. Assume
that $Z_{Min}$ and $(e_i^{Min})$ are the minimal space and minimal
basis with respect to $(x_i,f_i)$, respectively.

Then the following conditions are equivalent.
\begin{enumerate}
\item[a)]  Every normalized block sequence of $(x_i)$ is weakly null.
\item[b)]
\begin{enumerate}
\item[i)]
$(x_i,f_i)$ is locally shrinking.
\item[ii)]
If $(u_n)\subset B_X$ with $\lim\limits_{n\to\infty} f_m(u_n)=0$
 for all $m\in\N$, then $(u_n)$ is weakly null.
 \end{enumerate}
\item[c)] \ \ $(x_i,f_i)$ is locally shrinking and pre-shrinking.
\item[d)]
\begin{enumerate}
\item[i)]
$(x_i,f_i)$ is locally shrinking.
\item[ii)] $X^*=\overline{\text{\rm span}}(f_i:i\in\N)$.
 \end{enumerate}
 \item[e)]
\begin{enumerate}
\item[i)]
$(x_i,f_i)$ is locally shrinking.
\item[ii)] $(e_i^{Min})$ is a shrinking basis of $Z_{Min}$.
\end{enumerate}
\end{enumerate}
\end{thm}

\begin{thm}\label{cor.bdd}
Let $(x_i,f_i)$ be a frame of a Banach space $X$. Then the following
conditions are equivalent.
\begin{enumerate}
\item[a)] $(x_i,f_i)$ locally shrinking and for all $x^{**}\in X^{**}$,
 $\| x^{**}|_{\text{\rm span}(f:i:i\ge n)}\|\to0$, if $n\to\infty$ .
\item[b)] $(x_i,f_i)$ is locally shrinking, locally boundedly complete and pre-boundedly complete.
\item[c)]\begin{enumerate}
\item[i)] $(x_i,f_i)$ is locally shrinking and locally boundedly complete.
\item[ii)] For every $x^{**}\in X^{**}$, $\sum  x^{**}(f_i)
x_i$ converges under the topology $\sigma(X,
\overline{\textrm{span}(f_i:i\in\N)})$.
\end{enumerate}
\item[d)]\begin{enumerate}
\item[i)] $(x_i,f_i)$ is locally shrinking and locally boundedly complete.
\item[ii)] $X$ is isomorphic to
$\overline{\textrm{span}(f_i:i\in\N)}^*$ under the natural canonical
map.
\end{enumerate}
\item[e)]\begin{enumerate}
\item[i)] $(x_i,f_i)$ is locally shrinking and locally boundedly complete.
\item[ii)]  $(e^{Max}_i)$ is a boundedly complete basis of $Z_{Max}$.
\end{enumerate}
\end{enumerate}
\end{thm}

For the above main theorems, we need the following results.

\begin{prop}\label{pre.pre}
\begin{enumerate}
\item[(a)] Every frame  satisfying (a) of Theorem \ref{shrinking.1}  is pre-shrinking.
\item[(b)] Every frame  satisfying (a) of Theorem \ref{cor.bdd}  is
pre-boundedly complete.
\end{enumerate}
\end{prop}
\begin{proof}
Assume that $(x_i,f_i)$ is a frame of a Banach space $X$.

(a) Notice that every normalized block sequence of $(x_i)$ is weakly
null if and only if for all $f\in X^*$,
$\|f|_{\textrm{span}(x_i:i\ge n)}\|\to 0$, as $n\to\infty$. This easily
 implies our claim  by Proposition \ref{basic.pp1} (b) and
(c).

\noindent
(b) For $ m\le n$  in $\N$  we have
\begin{align}
\Big\| \sum_{i=m}^n x^{**}(f_i) x_i\Big\|&= \sup_{f\in
B_{X^*}}\Big|\sum_{i=m}^n x^{**}(f_i) f(x_i)\Big|\\
&=\sup_{f\in B_{X^*}}x^{**}\Big(\sum_{i=m}^n  f(x_i) f_i\Big)\notag\\
&\le \sup\limits_{ g \in \, \text{\rm span}(f_i:i\ge m), \|g\|\le K
}x^{**}(g) =
 K \, \|x^{**}|_{\text{\rm span}(f_i:i\ge m)}\|,\notag
\end{align}
where $K$ is the projection constant of
$(x_i,f_i)$.
\end{proof}

\begin{prop}\label{bdd.2}
Let $(x_i,f_i)$ is a Schauder frame of a Banach space $X$. Assume
that $Z$ is an associated space with an associated basis $(z_i)$ to
$(x_i,f_i)$.
\begin{enumerate}
\item[a)] If $(z_i)$ is shrinking, then $(x_i,f_i)$ is pre-shrinking.
\item[b)]
If $(z_i)$ is boundedly complete, then $(x_i,f_i)$ is pre-boundedly
complete.
\end{enumerate}
\end{prop}
\begin{proof}
Assume that $S$ and $T$ are the corresponding associated
reconstruction and decomposition operators, respectively. By
Proposition \ref{P:2.7a}, $S(z_i)=x_i$ and $T^*(z_i^*)=f_i$ for all
$i\in\N$.

(a) If $(z_i)$ is shrinking, we have
\begin{align}
f=T^*S^*(f)=T^*\Big( \sum \langle S^*(f), z_i\rangle z_i^* \Big)=
\sum \langle f, S(z_i)\rangle T^*(z_i^*)=\sum f(x_i) f_i,
\end{align}
which proves our claim.

\noindent
(b) For any $x^{**}\in X^{**}$ and $m,n\in\N$ with $m\le n$,
\begin{align}\label{bdd.3}
\Big\|\sum\limits_{i=m}^n
x^{**}(f_i)x_i\Big\|&=\Big\|\sum\limits_{i=m}^n
x^{**}(T^*(z_i^*))S(z_i)\Big\|=\Big\|S\Big(\sum\limits_{i=m}^n
T^{**}(x^{**})(z_i^*)z_i\Big)\Big\|   \\
&\leq\|S\|\cdot\Big\|\sum\limits_{i=m}^n
T^{**}(x^{**})(z_i^*)z_i\Big\|.\notag
\end{align}
Since $(z_i)$ is boundedly complete,
 $\sum\limits_{i=1}^\infty T^{**}(x^{**})(z_i^*)z_i$
converges, by (\ref{bdd.3}), so does $\sum\limits_{i=1}^\infty
x^{**}(f_i)x_i$, which completes the proof.
\end{proof}

\begin{prop}\label{shr.bdd}
Let $(x_i,f_i)$ be a Schauder frame of a Banach space $X$.
\begin{enumerate}
\item[a)] Assume that $Z_{Min}$ and $(e_i^{Min})$ are the minimal space
and minimal basis with respect to $(x_i,f_i)$, respectively. % Then
If $(x_i,f_i)$  satisfies (a) of Theorem \ref{shrinking.1}, then $(e_i^{Min})$ is shrinking.

\item[b)] Assume that $Z_{Max}$ are the maximal space
with the maximal basis $(e_i^{Max})$ with respect to $(x_i,f_i)$. If $(x_i,f_i)$ satisfies (a) of Theorem \ref{cor.bdd} , then $(e_i^{Max})$ is
boundedly complete.
\end{enumerate}
\end{prop}

For the proof of Proposition \ref{shr.bdd}, we will  need the
following result, which is a slight variation of Lemma 2.10 of
\cite{OS}.
\begin{lem}\label{L:3.2}
Let $X$ be a Banach space and a sequence $(x_i)\subset X\setminus
\{0\}$, and let $Z_{Min}$ and $(e^{Min}_i)$ be the associated minimal space and basis, respectively.
\begin{enumerate}
\item[a)] Let $(y_i)\subset B_{Z_{Min}}$ be a block basis of $(e_i^{Min})$
on $Z_{Min}$. Assume that the sequence $(w_i)=(S_{Min}(y_i))$ is  a
semi-normalized basic sequence in $X$. Then for $(a_i)\in \coo$,
$$\Big\|\sum a_i w_i\Big\|
\le\Big\|\sum a_i y_i\Big\| \le\Big(\frac{2K}a+K\Big) \Big\|\sum
a_iw_i\Big\|,$$ where $K$ is the projection constant of $(w_i)$ and
$a:=\inf\limits_{i\in\N} \|w_i\|$.
\item[b)] If every normalized block sequence of $(x_i)$ is weakly
null, then $(e_i^{Min})$ is shrinking.
\end{enumerate}
\end{lem}
\begin{proof}
 Let $S_{Min}:Z_{Min}\to X$ be defined as in Definition \ref{minimal.def}.

\noindent
(a) For $i\in\N$, write
$$y_i=\sum_{j=k_{i-1}+1}^{k_i} \beta_j^{(i)} e_j^{Min},\text{ with $0=k_0<k_1<k_2\ldots$ and $\beta_j^{(i)}\in\R$, for $i,j\in\N$},$$
and set
$$w_i=S_{Min}(y_i)=\sum_{j=k_{i-1}+1}^{k_i} \beta_j^{(i)} x_j.$$

Let $(a_i)\in\coo$. We use the definition of $Z_{Min}$ to find
  $1\le i_1\le i_2+1$ and $\ell_1\in[k_{i_1-1}+1,k_{i_1}]$
 and $\ell_2\in[k_{i_2}+1,k_{i_2+1}]$  in $\N$ so that, when $i_1\le
 i_2-1$,
\begin{align*}
\Big\|\sum a_i w_i\Big\|&\le \Big\|\sum a_i y_i\Big\|   \text{\ \ (Since $\|S_{Min}\|\le 1$)}\\
&=\Big\|a_{i_1} \sum_{j=\ell_1}^{k_{i_1}} \beta_j^{(i_1)} x_j
 +\sum_{s=i_1+1}^{i_2} a_s w_s+a_{i_2+1}\sum_{j=k_{i_2}+1}^{\ell_2} \beta_j^{(i_2)} x_j\Big\|\\
&\le \Big\|a_{i_1} \sum_{j=\ell_1}^{k_{i_1}} \beta_j^{(i_1)} x_j
\Big\|
 + \Big\|\sum_{s=i_1+1}^{i_2} a_s w_s \Big\|+ \Big\|a_{i_2+1}\sum_{j=k_{i_2}+1}^{\ell_2} \beta_j^{(i_2)} x_j\Big\|\\
&\le |a_{i_1}|\| y_{i_1}\|+ |a_{i_2+1}|\| y_{i_2+1}\|+K\Big\|\sum a_i w_i\Big\|\\
&\le |a_{i_1}|+ |a_{i_2+1}|+K\Big\|\sum a_i w_i\Big\|
 \le \Big(\frac{2K}a+K\Big)\Big\|\sum a_i w_i\Big\|.
\end{align*}
The other two cases $i_1=i_2$ and $i_1=i_2+1$ can be obtained in
similar way.

\noindent
(b) Assume that $(y_i)$ is a normalized block sequence of
$(e_i^{Min})$.
 For $i\in\N$, we write
$$y_i=\sum_{j=k_{i-1}+1}^{k_i} a_j e_j^{Min},\text{ with $0=k_0<k_1<k_2\ldots$ and $a_j\in\R$}.$$
Then, by definition of the space $S_{Min}$, $(S_{Min}(y_i))$ is a
bounded block sequence of $(x_i)$. It is enough to show that $(y_i)$
has a weakly null subsequence.

If $\liminf\limits_{i\to\infty}\|S_{Min}(y_i)\|>0$, then our claim
follows from (a). In the case that
$\lim\limits_{i\to\infty}\|S_{Min}(y_i)\|=0$, we use the definition
of $Z_{Min}$ to find
 $k_0<m_1\le n_1\le k_1<m_2\le n_2<\ldots$ so that for all $i\in\N$,
$1=\|y_i\|=\big\|\sum\limits_{j=m_i}^{n_i} a_i x_i\big\|.$ Thus, by
(a), the sequences $(w_i^{(1)})$ and  $(w_i^{(2)})$  with
 $$w_i^{(1)}=\sum_{j=m_i}^{n_i} a_j x_j\text{ and }
w_i^{(2)}=S_{Min}(y_i)- \sum_{j=m_i}^{n_i} a_j
x_j=\sum_{j=k_{i-1}+1}^{k_i} a_j x_j- \sum_{j=m_i}^{n_i} a_j x_j \,
\text{ for $i\in\N$},$$ both can, after passing to a further
subsequence, be assumed to be semi-normalized and, by hypothesis,
are weakly null, which implies that we can, after passing to a
subsequence again, also assume that they are basic. Claim (a)
implies that the sequences
 $(y_i^{(1)})$ and  $(y_i^{(2)})$  with
 $$y_i^{(1)}=\sum_{j=m_i}^{n_i} a_j e_j^{Min}\text{ and }
y_i^{(2)}=\sum_{j=k_{i-1}}^{k_i} a_j e_j^{Min}- \sum_{j=m_i}^{n_i}
a_j e_j^{Min} \, \text{ for $i\in\N$},$$ are weakly null in
$Z_{Min}$, which implies that $(y_i)$ is weakly null.
\end{proof}

\begin{proof}[Proof of Proposition \ref{shr.bdd}]
(a) It can be directly obtained by Lemma \ref{L:3.2} (b).

\noindent (b) Denote by $(e_i^*)$ the coordinate functionals of
$(e_i^{Max})$. Since $(x_i,f_i)$ is boundedly complete Proposition
\ref{associated.maximal.pp1} (c)  yields that $Z_{Max}$ is  an
associated space. Let $T_{Max}: X\to Z_{Max}$ be the  associated
decomposition operator, and recall that by Proposition \ref{P:2.7a},
$T_{Max}^*(e_i^*)=f_i$, for $i\in\N$. Then for any $(a_i)\in \coo$,
\begin{equation}
\max_{m\le n}\Big\|\sum_{i=m}^n a_i f_i\Big\|=\max_{m\le
n}\Big\|\sum_{i=m}^n a_i T_{Max}^*(e_i^*)\Big\|\le \|T_{Max}^*\|
\max_{m\le n}\Big\|\sum_{i=m}^n a_i e_i^*\Big\|\le K \|T_{Max}^*\|
\Big\|\sum a_i e_i^*\Big\|,
\end{equation}
where $K$ is the projection constant of $(e_i^*)$. Moreover,
\begin{eqnarray*}
\Big\|\sum a_i e_i^*\Big\|&=&\sup_{\stackrel{(b_i)\in\coo}{\|\sum
b_i e_i^{Max}\|\leq 1}} \Big|\sum a_i b_i\Big|\\ &\leq&
\sup_{\stackrel{(b_i)\in\coo}{\|\sum b_i e_i^{Max}\|\leq 1}} \quad
\sup_{\stackrel{(c_i)\in\coo}{\max\limits_{m\le
n}\|\sum\limits_{i=m}^n c_i f_i\|\leq \max\limits_{m\le
n}\|\sum\limits_{i=m}^n a_i f_i\|}} \Big|\sum c_i
b_i\Big| \\
&=&\max\limits_{m\le n}\Big\|\sum\limits_{i=m}^n a_i f_i\Big\|
\sup_{\stackrel{(b_i)\in\coo}{\|\sum b_i e_i^{Max}\|\leq 1}} \ \
\sup_{\stackrel{(c_i)\in\coo}{\max\limits_{m\le
n}\|\sum\limits_{i=m}^n c_i f_i\|\leq 1}} \Big|\sum c_i
b_i\Big|\\
&=&\max\limits_{m\le n}\Big\|\sum\limits_{i=m}^n a_i f_i\Big\|
\sup_{\stackrel{(b_i)\in\coo}{\|\sum b_i e_i^{Max}\|\leq
1}}\Big\|\sum b_i e_i^{Max}\Big\|  \leq \max\limits_{m\le
n}\Big\|\sum\limits_{i=m}^n a_i f_i\Big\|.
\end{eqnarray*}
Thus, $(e_i^*)$ is equivalent
 to the minimal basis with respect to $(f_i)\subset X^*$. By Proposition \ref{P:2.5b}
  $(f_i,x_i)$ is a frame for $\overline{\text{span}{(f_i:i\in\N})}$.
 $(e_i^{Min})$ with respect to $(f_i)$ in
$X^*$ constructed in Lemma \ref{L:3.2}. Since  by assumption
  $\| x^{**}|_{\text{\rm span}(f:i:i\ge n)}\|\to0$, if $n\to\infty$,
every normalized block
sequence of $(f_i)$ is weakly null. Therefore Lemma \ref{L:3.2} (b) yields that
 $(e_i^*)$  is shrinking. Thus, $(e_i^{Max})$
is boundedly complete, which proves our claim.
\end{proof}

We are now  ready to present a proof of our main theorems:

\begin{proof}[Proof of Theorem \ref{shrinking.1}]

\noindent (a)$\Rightarrow$(b) It is clear that (a) implies (b)(i),
while (b)(ii) follows from
 (a) and the fact that the frame representation (\ref{def.frame.eq1}) implies that
 every sequence $(u_n)\subset B_X$ for which $\lim\limits_{n\to\infty}f_i(u_n)=0$, whenever $i\in\N$,
 has a subsequence which is an arbitrary small perturbation of a block sequence of $(x_i)$ in $B_X$.

\noindent
(b)$\Rightarrow$(c) By Proposition \ref{basic.pp1} (c), every $f\in
X^*$ can be written as
 $$f=w^*-\lim\limits_{n\to\infty}\sum_{i=1}^n f(x_i) f_i.$$
If for some $f$, this sum did not converge in norm, we could find a
sequence $(u_k)\subset B_X$ and $m_1\le n_1<m_2\le n_2< . . .$ in
$\N$ and $\vp>0$ so that for all $k\in\N$,
\begin{equation}\label{E:3.1.1}
f\Big(\sum_{i=m_k}^{n_k} f_i(u_k) \, x_i\Big)=\sum_{i=m_k}^{n_k}
f(x_i) f_i(u_k)\ge \frac{1}{2}\Big\|\sum_{i=m_k}^{n_k} f(x_i)
f_i\Big\| \ge \frac{\vp}{2}.\end{equation} By Proposition
\ref{basic.pp1} (b), $(\tilde u_k)\subset K\cdot B_X$, where $\tilde
u_k=\sum\limits_{i=m_k}^{n_k} x_i f_i(u_k)$,  for $k\in\N$.
Thus, $\tilde u_k$ is a bounded block sequence of $(x_i)$ , which contradicts (b)(ii)

\noindent
(c)$\Rightarrow$(d) trivial.

\noindent
(d)$\Rightarrow$(a) by Proposition \ref{self}. Thus we verified
(a)$\Leftrightarrow$(b)$\Leftrightarrow$(c)$\Leftrightarrow$(d).

\noindent
(a)$\Leftrightarrow$(e) by Proposition \ref{shr.bdd} (a).

\noindent
(e)$\Leftrightarrow$(c)  by
Proposition \ref{bdd.2} (a).
\end{proof}

\begin{proof}[Proof of Theorem \ref{cor.bdd}]

\noindent
(a)$\Rightarrow$(b)  by Proposition \ref{pre.pre}.

\noindent
(b)$\Rightarrow$(c) trivial.

\noindent
(c)$\Rightarrow$(d) Let $Y=\overline{\textrm{span}(f_i:i\in\N)}$.
Define $J: X\rightarrow Y^*$ by $J(x): f \mapsto f(x)$, which is the
natural canonical map. Then we have $$\|J(x)\|=\sup\limits_{f\in
B_Y} |J(x)f|=\sup\limits_{f\in B_Y} |f(x)|\leq \|x\|,$$ which
implies that $J$ is a bounded linear operator. Next we will show
that $J$ is bijective. Since, by Proposition \ref{norming}, $Y$ is a
norming set of $X$, $J$ is injective. On the other hand, any
$y^{*}\in Y^*$, can,  by the Hahn-Banach Theorem, be extended it to an element  $x^{**}\in X^{**}$.
Then by hypothesis, there is an $x\in X$ such that
$x=\lim\limits_{n\rightarrow\infty}\sum\limits_{i=1}^n x^{**}(f_i)
x_i$ under the topology $\sigma(X, Y)$. Thus, for any $f\in Y$,
\begin{equation}
J(x)(f)=f(x)=\lim_{n\rightarrow\infty} f\Big(\sum_{i=1}^n
x^{**}(f_i)
x_i\Big)=\lim_{n\rightarrow\infty}x^{**}\Big(\sum_{i=1}^n f(x_i)
f_i\Big)=x^{**}(f),
\end{equation}
which implies that $J$ is surjective. Then by the Banach Open
Mapping Principle, $J$ is an isomorphism from $X$ onto $Y^*$.

\noindent
(d)$\Rightarrow$(a) Let $x^{**}\in X^{**}$ and put $f^*=x^{**}|_Y\in Y^*$
(i.e.$f^*(f)=x^{**}(f)$ for $f\in Y$). By assumption (d) there is an $x\in X$ so that
$f(x)= f^*(f)=x^{**}(f)$ for ll $f\in Y$. Thus
 (a) follows from Proposition \ref{self}.

 Note we have now verified the equivalences
 (a)$\Leftrightarrow$(b)$\Leftrightarrow$(c)$\Leftrightarrow$(d).

\noindent (a)$\Rightarrow$(e) by Proposition \ref{shr.bdd}.

\noindent (e)$\Rightarrow$ (b) by Proposition \ref{bdd.2} (b) and Theorem \ref{associated.maximal.pp1} (c)
\end{proof}

\begin{ex}\label{Ex:3.3}
 The following example shows that there is
a semi-normalized tight Hilbert frame for $\ell_2$ satisfying
(b)(ii) and (d)(ii) in Proposition \ref{shrinking.1} but not
condition (b)(i).

Choose $c>0$ and $(c_i)\subset(0,1)$ so that
\begin{equation}\label{E:3.3.1}
 c^2+\sum c_i^2=1\text{ and } \sum c_i=\infty
\end{equation}
In $\ell_2$  put  $x_1=ce_1$ and for $i\in\N$
$$x_{2i}=\frac1{\sqrt2}e_{i+1} +\frac{c_i}{\sqrt2} e_1\text{ and }x_{2i+1}=\frac1{\sqrt2}e_{i+1}
 -\frac{c_i}{\sqrt2} e_1.$$
It follows for any $x=\sum a_i e_i\in\ell_2$ that
\begin{align*}
\sum_{i=1}^\infty \langle x_i, x\rangle^2&= c^2
a_1^2+\frac12\sum_{j=2}^\infty
 (a_j+c_{j-1} a_1)^2+ (a_j-c_{j-1} a_1)^2\\
&= c^2a_1^2+\sum_{j=2}^\infty a_j^2+a_1^2\sum_{j=1}^\infty
c_j^2=\|x\|^2,
\end{align*}
Thus, $(x_i)$ is a tight frame, which implies (b)(ii), (c)(ii) and
(d)(ii).

Using the second part of (\ref{E:3.3.1}) we can choose
$0=n_0<n_1<n_2<\ldots$ so that
$$\lim_{i\to\infty} y_i=e_1, \text{ where }y_i=\sum_{j=n_{i-1}+1}^{n_i}x_{2j}- x_{2j+1} \text{ for }i\in\N,$$
which implies that (b)(i) is not satisfied.
\end{ex}

\begin{prop}\label{pre.shr}
Let $(x_i,f_i)$ be a Schauder frame of a Banach space $X$. Then the
following conditions are equivalent:
\begin{enumerate}
\item[a)] $(x_i,f_i)$ is a pre-shrinking Schauder frame of $X$.
\item[b)] $(f_i,x_i)$ is a pre-boundedly complete Schauder frame of $X^*.$
\item[c)] $(f_i,x_i)$ is a pre-boundedly complete Schauder frame of $\overline{\textrm{span}(f_i:i\in\N)}.$
\end{enumerate}
\end{prop}
\begin{proof} (a)$\Rightarrow$(b) Assume that
$(f_i,x_i)$ is a Schauder frame of $X^*.$ For any $x^{***}\in
X^{***}$, $x^{***}|_{X}$ is a continuous linear functional on $X$.
Then $\sum\limits_{i=1}^\infty
x^{***}(x_i)f_i=\sum\limits_{i=1}^\infty x^{***}|_X (x_i)f_i$
converges in $X^*$, which completes the claim.

\noindent
(b)$\Rightarrow$(c) is trivial.

\noindent
(c)$\Rightarrow$(b) Let $Y=\overline{\textrm{span}(f_i:i\in\N)}.$
 and let  $f\in X^*$.  By Proposition \ref{norming}. $X$ can be
isomorphically embedded into $Y^*$ under the natural canonical map.
By the Hahn-Banach Theorem, extend $f$ to an element in $Y^{**}$ and, thus, assumption  (c) yields
 that $\sum\limits_{i=1}^\infty f(x_i)f_i$ converges in $Y.$ Since this series converges
  in $w^*$ to $f$ by
Proposition \ref{basic.pp1} this completes the proof.\end{proof}

\begin{prop}\label{reflexive}
 Let $(x_i,f_i)$ be a Schauder frame of a Banach space $X$.

If $(x_i,f_i)$ is pre-shrinking and pre-boundedly complete, then
$X$ is reflexive.
\end{prop}
\begin{proof} Since $(x_i,f_i)$ is pre-shrinking we can write every $f\in X^*$ as
$f=\sum f(x_i) f_i$. Since $(x_i,f_i)$ pre-boundedly complete we can choose for each
$x^{**}\in X^{**}$ an $x\in X$ so that
$x=\sum x^{**}(f_i) x_i$. Thus for any $f\in X^*$
 $$x^{**}(f)=\sum  f(x_i)x^{**}(f_i)=f(x),$$
 which proves our claim.
\end{proof}

\section{Unconditional Schauder Frames}

The following result extends James \cite{Ja}  well known result on
unconditional bases to unconditional frames.

\begin{thm}\label{un.bdd}
Let $(x_i, f_i)$ be an unconditional and locally shrinking Schauder
frame of a Banach space $X$.
\begin{enumerate}
\item[(a)]
If $(x_i,f_i)$ is not pre-boundedly complete, then $X$ contains an
isomorphic copy of $c_0.$
\item[(b)]
If $(x_i, f_i)$ is not shrinking, then $X$ contains an isomorphic
copy of $\ell_1.$
\end{enumerate}
\end{thm}

Then  by Proposition \ref{reflexive} and Theorem \ref{un.bdd}, we
obtain  Theorem C.

For the proof, we need the following lemma.

\begin{lem}\label{Basic subseq lem}
Let $X$ be a separable Banach space and $(x_i,f_i)\subset X\times
X^*$ be a locally shrinking Schauder frame of $X$ with the
projection operator $K$. Let $Y$ be a finite-dimensional subspace of
$X$. Then for every $\vp>0$, there exists $N\in \N$ such that
$\|y\|\leq (K+\vp)\|y+ x\|$ whenever $x\in \text{\rm span}(x_i:i\geq
N)$ and $y\in Y$.

\end{lem}
\begin{proof} W.l.o.g. $\vp<1/2$. Let $(y_i)_{i=1}^n$ be an $\frac{\vp}{8K^2}$-net of $S_{Y}$,
and $(x^*_i)_{i=1}^n\subset S_{X^*}$ with $x_i^*(y_i)=1$ for $1\le
i\le n$. For large enough $k$ it follows
 that $(\tilde x^*_i)_{i=1}^n$, with $\tilde x^*_i=\sum_{j=1}^k x^*_i(x_j) f_j$, $i=1,2,\ldots n$,
 satisfies that
 $$\|\tilde x^*_i\|\le K \ \ (1\le i\le n) \ \ \mbox{ and } \,
 \max_{1\le i\le n} |\tilde x_i^*(y)|\ge 1-\frac{\vp}{4K} \ \ \mbox{ for all } y\in S_Y.$$ It follows that
 $\|\tilde x_j\|\le K$, for $j=1,2,\ldots n$.
 Using our assumption that $(x_i,f_i)$ is locally shrinking we can choose $N\in\N$,
 so that $\|\tilde x^*_i|_{\text{span}(x_j:j\ge N)}\|\le \frac{\vp}{8K}$.

 If $y\in Y$  and $x\in \text{\rm span}(x_i:i\geq N)$  , then either $\|x\|\ge 2 \|y\|$, in which
 case $\|y+x\|\ge \|x\|-\|y\|\ge \|y\|$.
 Or $\|x\|\le 2 \|y\|$, and then
 $$K \,\|y+x\|\ge \max_{i\le n} |\tilde x^*_i(y+x)|\ge \Big(1-\frac{\vp}{4K}\Big)\|y\|
 -\frac{\vp}{8K} \|x\|\ge \Big(1-\frac{\vp}{2K}\Big)\|y\|\ge \frac{\|y\|}{1+\vp/K}.$$
\end{proof}

\begin{cor}\label{Basic subseq}
Let $X$ be a separable Banach space and $(x_i,f_i)\subset X\times
X^*$ be a locally shrinking Schauder frame of $X$ with the
projection operator $K$. Then for every
normalized block sequence $(u_i)$ of $(x_i)$ and every $\epsilon >0$,
 there is a basic subsequence of $(u_i)$ whose basis constant $K_b$
is not larger than $K+\epsilon$.
\end{cor}
\begin{proof}

Using at each step Lemma \ref{Basic subseq lem}
 we can choose a subsequence basis $(v_i)$ of $(u_i)$, so that  for all $N\in\N$
$$\|y+ x\|\geq \frac{\|y\|}{K+\vp} \ \ \mbox{ for all } y\in \text{\rm span}(v_i:i\leq
N) \mbox{ and } x\in \text{\rm span}(v_i:i\geq N+1).$$ It follows
then, that $(v_i)$ is basic and its basis constant does not exceed
$K+\vp$.
\end{proof}

\begin{lem}\label{uncond.block} Assume that $(x_i,f_i)$ is an unconditional  and locally shrinking frame for a Banach space $X$. Let $K_u$ be the constant of unconditionality of $(x_i,f_i)$ and let $(u_i)$ be a block basis of
$(x_i)$. For any $\vp>0$ there is a subsequence $(v_i)$ of $(u_i)$ which is $K_u+\vp$ unconditional.
\end{lem}
\begin{proof} W.l.o.g. $\|x_n\|=1$, for $n\in\N$,  otherwise replace $x_n$ by $x_n/\|x_n\|$ and
$f_n$ by $f_n\|x_n\|$.  By Corollary \ref{Basic subseq} we can assume that $(u_i)$ is $2K_u$-basic
(note that the projection constant of $(x_i,f_i)$ is at most $K_u$.

Let $(\delta_i)\subset (0,1)$ with $\sum_{j>i}\delta_j< \delta_i$, $i\in\N$, and $\sum \delta_i<
\vp/8K^2_u$.
Then we choose recursively   increasing  sequences $(n_i)$ and $(k_i)$ in $\N$ so that
\begin{align}
\label{E1}&|f_s(u_{n_i})|<  \frac{\delta_i}{k_{i-1}}       \text{whenever $s\le k_{i-1}$},\text{ and }\\
\label{E2}&\Big\|\sum_{s=k_i}^N  f_s\Big( \sum_{j=1}^{i} \lambda_j u_{n_j}\Big) x_s\Big\| < \delta_{i+1}\text{ whenever $N\ge k_i$ and $(\lambda_j)_{j=1}^i \subset[-1,1]$}.
\end{align}
Indeed, assume $k_{i-1}$ was chosen ($k_0=1$). Since
 $(x_i,f_i)$ is locally shrinking,
we can choose $n_{i}$ so that \eqref{E1} is satisfied. Secondly, using the compactness of
the set $\{ \sum_{j=1}^i \lambda_j u_{n_j}:(\lambda_j)_{j=1}^i \subset[-1,1]\}$, we can choose
 $k_i$ so that \eqref{E2} is satisfied.

We are given now $(\lambda_i)\subset c_{00}$ with $\max |\lambda_i|=1$
and $(\vp_i)\subset\{-1,1\}$.
For $u=\sum \lambda_iu_{n_i}$ and $\ub=\sum \vp_i \lambda_i u_{n_i} $ we compute:
\begin{align*}
\|\ub\|&=\Big\|\sum_{s=1}^\infty f_s(\ub) x_s\Big\|\\
     & =\Big\|\sum_{i=1}^\infty \sum_{s=k_{i-1}}^{k_i-1} f_s(\ub) x_s\Big\|\\
      &\le K_u\Big\|\sum_{i=1}^\infty \vp_i \sum_{s=k_{i-1}}^{k_i-1} f_s(\ub) x_s\Big\|  \\
       &\le K_u\Big\|\sum_{i=1}^\infty  \sum_{s=k_{i-1}}^{k_i-1} \lambda_i f_s(u_{n_i}) x_s\Big\| \!+\!
           K_u\sum_{i=1}^\infty  \Big\|\sum_{s=k_{i-1}}^{k_i-1}
            f_s\Big( \sum_{j=1}^{i-1}  \vp_j\lambda_j  u_{n_j} \Big) x_s\Big\|\\
             & \qquad \!+\!K_u\sum_{i=1}^\infty \sum_{s=k_{i-1}}^{k_i-1} \sum_{j={i+1}}^\infty |f_s(u_{n_j})|    \\
             &   \le K_u\Big\|\sum_{i=1}^\infty  \sum_{s=k_{i-1}}^{k_i-1} \lambda_i f_s(u_{n_i}) x_s\Big\| +
               \frac\vp{8K_u}    +   K_u \sum_{i=1}^\infty  k_i \sum_{j=i+1}^\infty \frac{\delta_j}{k_i}\\
               &   \le K_u\Big\|\sum_{i=1}^\infty  \sum_{s=k_{i-1}}^{k_i-1} \lambda_i f_s(u_{n_i}) x_s\Big\| +
               \frac{\vp}{4K_u}.
\end{align*}
 By switching the role of $u$ and $\ub$, we compute also
 $$\Big\|\sum_{i=1}^\infty  \sum_{s=k_{i-1}}^{k_i-1} \lambda_i f_s(u_{n_i}) x_s\Big\| \le
 \Big\|\sum_{i=1}^\infty  \sum_{s=k_{i-1}}^{k_i-1} f_s(u) x_s\Big\| +\frac{\vp}{4K_u}= \|u\| +\frac{\vp}{4K_u}.$$
 Since  the basis constant of $(u_i)$  does not exceed $2K_u$
it follows that $\|\ub\|$, $\|u\|\ge \frac1{2K_u}$.
and thus
$$\|\ub\|\le \|u\|+\frac{\vp}{2K_u}\le( K_u+\vp)\|u\|,$$
which proves our claim.
\end{proof}

\begin{proof}[Proof of Theorem \ref{un.bdd}.] (a)
By  assumption, there is some $x_0^{**}\in S_{X^{**}}$ such that
$\sum_{i=1}^n x_0^{**}(f_i) x_i$ does not converge. By the Cauchy
criterion, there are $\delta>0$ and natural numbers
$p_1<q_1<p_2<q_2\cdot\cdot\cdot$ such that for
$u_j=\sum_{i=p_j}^{q_j} x_0^{**}(f_i) x_i$ we have $\|u_j\|\geq\delta$
for every $j$. By Corollary \ref{Basic subseq}, we can find a basic
subsequence $(u_{n_j})$ of $(u_j)$ with the basis constant $C>1$.
Then for every sequence $(\lambda_j)_{j=1}^m$ of scalars and every
$i\in\{1,...,m\}$, we have $ \left\|\sum_{j=1}^m \lambda_j
u_{n_j}\right\|\geq\frac{1}{2C}\|\lambda_i u_{n_i}\|\geq
\frac{\delta}{2C}|\lambda_i|. $ That is, $\|\sum_{j=1}^m \lambda_j
u_{n_j}\|\geq \frac{\delta}{2C}\|(\lambda_j)\|_\infty$.

Recall that the unconditional constant of $(x_i,f_i)$ is defined by
\begin{equation*}K_u=\sup_{x\in B_X} \sup_{(\vp_i)\subset\{\pm
1\}}\Big\|\sum_{i=1}^\infty \vp_i f_i(x) x_i\Big\|=
\sup_{x\in B_X} \sup_{(\lambda_i)\subset[-1,1]}\Big\|\sum_{i=1}^\infty \lambda_i f_i(x) x_i\Big\|<\infty
\end{equation*}
(the second ``$=$'' follows from a simple convexity argument).
Secondly we compute
\begin{align*}
\sup_{(\lambda_i)\in c_{00}\cap[-1,1]^\N} \big\| \sum \lambda_i u_i\big\|&
=\sup_{(\lambda_i)\in c_{00}\cap[-1,1]^\N} \big\| \sum_i \sum_{s=p_i}^{q_i} \lambda_i x^{**}_0(f_i)\big\|\\
&\le \sup_{x^{**}\in B_{X^{**}}} \sup_{(\lambda_s)\in c_{00}\cap[-1,1]^\N} \big\| \sum_s  \lambda_s x^{**}(f_i) x_i\big\| \\
 &= \sup_{x\in B_{X}} \sup_{(\lambda_s)\in c_{00}\cap[-1,1]^\N} \big\| \sum \lambda_s f_s(x) x_i\big\| =K_u.
\end{align*}

 \noindent (b) Since $(x_i, f_i)$ is not shrinking,  there exists $f\in S_{X^*}$
 and a normalized  block basis $(u_n)$ of $(x_n)$ and a $\delta>0$,
 so that $f(u_n)\ge \delta$, for $n\in\N$.  Since by Lemma \ref{uncond.block} we can assume that
 $(u_n)$ is  $2K_u$-unconditional, it follows  $(\lambda_i)\in c_{00}$ that
    $$\Big\|\sum \lambda_i u_i\Big\|\ge \frac1{2K_u}
  \Big\|\sum |\lambda_i| u_i\Big\|\ge f\Big(\sum |\lambda_i| u_i\Big)\ge\frac{\delta}{K_u}\sum |\lambda_i| .
  $$
\end{proof}

\vskip3mm

\noindent\textbf{Acknowledgment.}

This paper forms a portion of the author's doctoral dissertation,
which is being prepared at Texas A$\&$M University and Nankai
University under the direction of Thomas Schlumprecht and Guanggui
Ding. The author thanks Dr. Schlumprecht and Ding for their
invaluable help, guidance and patience.


\begin{thebibliography}{CPY74}

\bibitem[AG]{AG} A. Alroubi, K. Gr$\ddot{\mathrm{o}}$chenig,
\emph{Nonuniform sampling and reconstruction in shift invariant
spaces}, SIAM Rev. \textbf{43} (2001), 585--620.

\bibitem[CL]{CL} D. Carando, S. Lassalle, \emph{Duality,
 reflexivity and atomic decompositions in Banach spaces},
 Studia Math. \textbf{191} (2009), 67--80.

\bibitem[CDOSZ]{CDOSZ} P. G. Casazza, S. J. Dilworth, E. Odell, Th. Schlumprecht, A. Zsak,
{\em Coefficient Quantization for Frames in Banach Spaces}, J. Math.
Anal. Appl. \textbf{348} (2008), 66--86.

\bibitem[CFTW]{CFTW} P. G. Casazza, M. Fichus, J. C. Tremain, E.
Weber, The Kadison-Singer problem in mathematics and engineering: a
detailed account, Contemporary Math. 414 (2006), 299-356.

\bibitem[CHL]{CHL} P. G. Casazza, D. Han, and D. R. Larson,
 {\em Frames for Banach spaces}, The functional and harmonic analysis of wavelets and frames (San Antonio, TX, 1999),
Contemp. Math. {\bf 247} (1999), 149--182.

\bibitem[Ch]{Ch} O. Christensen, \emph{An introduction to frames and Riesz
bases}, Birkhauser (2003).

\bibitem[DE]{DE} D. L. Donoho, M. Elad, \emph{Optimally sparse
representations in general non orthogonal dictionaries via $\ell_1$
minimization}, Proc. Natl. Acad. Sci. USA \textbf{100} (2003),
2197--2202.

\bibitem[EF]{EF} Y. C. Eldar, G. D. Forney, \emph{Optimal tight frames and
quantum measurement}, IEEE Trans. Inform. Theory \textbf{48} (2002)
599--610.

\bibitem[Gr]{Gr} K. Gr$\ddot{\mathrm{o}}$chenig, Foundations of
Time-Frequency Analysis, Birkh$\ddot{\mathrm{a}}$user, Boston, 2000.

\bibitem[GH]{GH} K. Gr$\ddot{\mathrm{o}}$chenig, C. Heil, \emph{Modulation
spaces and pseudodifferential operators}, Integral Equations
Operator Theory \textbf{34} (1999) 439--457.

\bibitem[HL]{HL} D. Han, D. H. Larson, \emph{Frames, bases, and group
representation}, Mem. Amer. Math. Soc. \textbf{147} (2000), x+94 pp.

\bibitem[Ja]{Ja} R.~C.~James, \emph{Bases and reflexivity of Banach spaces}, Ann. of Math. \textbf{52} (1950), 518--527.


\bibitem[OS]{OS} E. Odell, Th. Schlumprecht, {\em A universal reflexive space for the
class of uniformly convex Banch spaces}, Math. Ann. \textbf{335}
(2006), no. 4, 901--916.



\end{thebibliography}
\end{document}